\documentclass{article}
\usepackage{graphicx} 
\usepackage{amsmath,amsthm,amssymb}
\usepackage{mathtools}
\usepackage[top=2.5cm, bottom=2.5cm, left=3cm, right=3cm]{geometry}
\newtheorem{theorem}{Theorem}  
\newtheorem{lemma}[theorem]{Lemma}     
\newtheorem{definition}{Definition}
\title{Pick-up Sticks and the General Fibonacci Numbers }
\author{Tian Cao Lin}
\date{January 2026}

\begin{document}

\maketitle
\begin{abstract}
    In the article by Edward et al. \cite{Sudbury2025}, it was shown that the probability that no three sticks randomly chosen from the unit interval can form a triangle equals the reciprocal of the product of the first $n$ Fibonacci numbers. The authors further suggested a generalization to higher \((k+1)\)-gons \((k\ge 4)\). 
    This note proves that, indeed, for any \(k\ge 2\), the probability that no $k+1$ of $n$ independent uniform $[0,1]$ lengths can form a $(k+1)$-gon is expressed as a product whose factors involve a $k$-step Fibonacci-type recurrence. The method follows closely the original argument of \cite{Sudbury2025}, while making explicit the algebraic structure that governs the general case.
\end{abstract}

\section{ Introduction}
The problem of random sticks is one of the classic problems in geometric probability with a long history, dating back to the ``broken stick problem'' posed by Lemoine in 1873~\cite{Lemoine1873}: a stick of unit length is broken at two random points, and one asks for the probability that the three resulting segments can form a triangle. Since then, many generalizations of this basic setting have been studied. D'Andrea and Gomez~\cite{DAngelis2006} extended the problem to the case of \(n\) pieces, proving that the probability that a stick broken randomly into \(n\) pieces can form an \(n\)-gon is \(1 - n/2^{n-1}\). More generally, Verreault~\cite{Verreault2022a, Verreault2022b} and Mukerjee~\cite{Mukerjee2024} systematically studied the ``\(k\)-gon'' version: given a stick broken into \(n\) pieces, what is the probability that every subset of \(k\) pieces can form a \(k\)-gon? Remarkably, these probability formulas are deeply connected with the \((k-1)\)-step Fibonacci numbers.

In a different direction, Petersen and Tenner~\cite{Petersen2020} introduced a model that is fundamentally different in its assumptions --- the \textbf{pick-up stick model}. In this model, the lengths of \(n\) sticks are drawn independently from the uniform distribution on \([0,1]\), and there is no constraint that they sum to one. They proved that the probability that \(n\) such sticks cannot form an \(n\)-gon is \(\frac{1}{(n-1)!}\)~\cite{Petersen2020}.

In the article \cite{Sudbury2025}, the authors proved that the probability that no three sticks can form a triangle is the reciprocal of the product of the first \(n\) Fibonacci numbers. They extended this method to quadrilaterals, where the probability expression involves the Tribonacci sequence and includes a correction term, and they discussed further extensions to general \(k+1\)-gons for \(k\geq2\). Their results hint at a general pattern: the avoidance of \(k+1\)-gons corresponds to a \(k\)-step Fibonacci-type recurrence. This note proves this generalization.

The method used in this paper is essentially the same as that in the article \cite{Sudbury2025} .  
 
\section{Main result}

We first denote
\begin{align} \label{matrix A_k}
     A_k=\begin{pmatrix}
1& 0&\dots&0&k-1\\
1& 0&\dots&0&k-2\\
0&1&\dots&0&k-3\\
\vdots&\vdots&&\vdots&\vdots\\
0& 0&\dots&1&0
 \end{pmatrix}\in \mathbb{R}^{k\times k}.
 \end{align} 
 and define the recurrence that will appear.

\begin{definition}
For a fixed integer $k\ge 2$, the sequence $\bigl\{F_k^{l}\bigr\}_{l=2-k}^{\infty}$ is called a \emph{$k$-step Fibonacci‑type recurrence} if it satisfies
\[
F_k^{l}=\sum_{j=1}^{k} F_k^{l-j},\qquad l= 2,3,\dots,
\]
with the initial conditions
\[
F_k^{1}=1,\quad F_k^{0}=F_k^{-1}=\dots=F_k^{-k+2}=0 .
\]
\end{definition}
In particular, when \(k=2\), the sequence \(\{F_2^{l}\}_{l=1}^{\infty}\) is the well-known Fibonacci sequence. \textbf{Theorem 1} is our the main result.

\begin{theorem}\label{Mainresult}
If \(n\) real numbers are chosen uniformly from the unit interval \([0, 1]\), for \(k\geq 3\),
the probability that no \(k+1\) of the numbers can form a \(k+1\)-gons  is
\[P_n^{(k)}=\prod_{l=1}^{n-k+2}\frac{1}{F^l_k} \prod_{i=1}^{k-2}\frac{1}{F^{n-i+1}_k-\sum_{j=1}^{k-i-1}(k-i-j)F_k^{n-k+j}},\]
where \(F_k^l\) is the \(l\)-th term in  \(k\)-step Fibonacci-type recurrence.
\end{theorem}

Before proving the theorem we state a lemma that captures the algebraic relations generated by the integration process.

\begin{lemma}\label{lemma A}
Let matrix \(A_k\) defined by (\ref{matrix A_k}), \(R^1=(1,1,\dots,1)^T\in \mathbb{R}^k\) and \(R^l=(R_1^l,R^l_2,\dots,R_k^l)^T\in\mathbb{R}^k\) satisfy
    \[R^l=A^{l-1}_kR^{1},l=1,2,3,\dots\]
Then, \[R_k^l=F_k^l,R_{k-1}^l=F^{l+1}_k\] and 
\[\quad R_{i}^l=F^{l+k-i}_k-\sum_{j=1}^{k-i-1}(k-i-j)F_k^{l+j-1},i=1,2,\dots,k-2.\]
 where \(F_k^l\) \(((l=1,2,3,\dots,)\) is \(k\)-step Fibonacci-type recurrence.
\end{lemma}
Now, assuming the lemma holds, we prove the main theorem.

\begin{proof}
Let $U_{(1)}\le U_{(2)}\le\dots\le U_{(n)}$ be the order statistics of $U_1,\dots,U_n$.  The condition that no $k+1$ lengths can form a $(k+1)$-gon is equivalent to
\begin{align}\label{first f}
\sum_{i=l}^{l+k-1} U_{(i)}\leq U_{(l+k)},\qquad \text{for all } 1\leq l\leq n-k.
\end{align}
A classical fact \cite{johnson1970continuous} states that the order statistics of a uniform sample can be represented via normalized sums of independent exponential variables, if \(x_1,\dots,x_{n+1}\stackrel{\mathrm{i.i.d.}}{\sim}\operatorname{Exp}(1)\) with density \(f(x)=e^{-x}\) and $S=x_1+\dots+x_{n+1}$, then
\[
U_{(i)}\stackrel{d}{=}\frac{x_1+\dots+x_i}{S},\qquad i=1,\dots,n .
\]
Substituting this into (\ref{first f}) yields

\[
\begin{aligned}
\sum_{l=0}^{k-1}\frac{x_{1}+\cdots+x_{i+l}}{S}&\leq\frac{x_{1}+\cdots+x_{i+k}}{S} \\
 (k-1)\sum_{ l=1}^ix_l+(k-2)&x_{i+1}+\dots +x_{i+k-2}\leq x_{i+k},
\end{aligned}
\]
where \(1\leq i\leq n-k\).  With \(x_{l}\) \((l=1,2,\dots,k)\) being unbounded between \(0\) and infinity.  

The successive inequalities define a nested region in \(\mathbb{R}^{n}\), and the required probability corresponds to the volume of this region under the exponential density. To evaluate it, we integrate successively from \(x_{n}\) down to \(x_{1}\), each step producing a factor determined by the recurrence relation that will lead to the \(k\) step Fibonacci pattern.

To make the subsequent formulas more concise, we introduce the notation \(L_j\), that
\begin{align} \label{simplnotations}
L_{j}= (k-1)\sum_{ l=1}^{j-k}x_l+(k-2)x_{j-k+1}+\dots +x_{j-2},
\end{align}
for \(j=k+1,k+2,\dots,n\). And \(L_j=0\) for \(j=1,2,\dots,k\). 

the probability we need to calculate is

\begin{align*}
    P_n^{(k)}=\int_{L_1}^\infty \int_{L_2}^\infty&\dots\int_{L_n}^\infty e^{-\sum_{i=1}^n x_i}dx_n\dots dx_2dx_1\\
    =\int_{L_1}^\infty\int_{L_2}^\infty&\cdots \int^\infty_{L_{n-1}} e^{-k\sum_{ l=1}^{n-k} x_l+(k-1)x_{n-k+1}+\dots +2x_{n-2}+x_{n-1}}dx_{n-1}\dots dx_2dx_1
\end{align*}
For convenience, we introduce the vector \( R^1,R^2\in \mathbb{R}^k \) to rewrite this integral, 

\begin{align*}
    P_n^{(k)}=\int_{L_1}^\infty \int_{L_2}^\infty&\dots\int_{L_n}^\infty e^{-R_1^1\sum_{l=1}^{n-k+1} x_l-R_2^1x_{n-k+2}-\dots -R_k^1x_n}dx_n\dots dx_2dx_1\\
    =\int_{L_1}^\infty\int_{L_2}^\infty&\cdots \int^\infty_{L_{n-1}} e^{-R_1^2\sum_{ l=1}^{n-k} x_{l}-R_2^2x_{n-k+1}-\dots -R_{k-1}^2x_{n-2}-R_k^2x_{n-1}}dx_{n-1}\dots dx_2dx_1
\end{align*}

Here, the last term \(x_{n-k}\) in the \(\sum\)-summation has been taken out separately to compensate for the missing integrated \(x_n\).
Let the integral obtained after \(j\) integrations be denoted \(I_j^{(k)}\), for example,
\begin{align*}
    I_1^{(k)}=&\int_{L_n}^\infty e^{-R_1^1(\sum_{t=1}^{n-k} x_t+x_{n-k+1})-R_2^1x_{n-k+2}-\dots -R_k^1x_n}dx_n
    \\=&\frac{1}{R_k^1}e^{-R_1^2\sum_{ t=1}^{n-k} x_{t}-R_2^2x_{n-k+1}-\dots -R_{k-1}^2x_{n-2}-R_k^2x_{n-1}}.
    \end{align*}
where \(L_n\) is given by formula (\ref{simplnotations}) and 
\[\left\{
\begin{aligned}
    R_1^2&=R_1^1+(k-1)R_k^1,\\
    R^2_2&=R_1^1+(k-2)R_k^1,\\
    R_3^2&=R_2^1+(k-3)R_k^1,\\
    \cdots\\
    R^2_k&=R_{k-1}^1.
\end{aligned}
\right.\]

Now extend the relationship between \( R^1 \) and \( R ^2\) to the general case, i.e., the properties that \(R^l\) possesses after \(l\) integrations. By induction, 
\begin{align*}
    I_l^{(k)}&=\prod_{t=1}^{l-1}\frac{1}{R^t_k}\int_{L_{n-l+1}}^\infty e^{-R_1^{l}(\sum_{t=1}^{n-l-k+1} x_t+x_{n-l-k+2})-R_2^{l}x_{n-l-k+3}-\dots -R_k^{l}x_{n-l+1}}dx_{n-l+1}
    \\&=\prod_{t=1}^l\frac{1}{R_k^t}e^{-R_1^{l+1}\sum_{ t=1}^{n-l-k+1} x_{t}-R_2^{l+1}x_{n-l-k+2}-\dots -R_{k-1}^{l+1}x_{n-l-1}-R_k^{l+1}x_{n-l}}.
\end{align*}
In the \(\Sigma\)-summation term of the second integrand, a factor is extracted outside, ensuring that the exponent always contains \(k\) terms.  Then 
\begin{align*}
R_1^{l+1}\sum_{ t=1}^{n-l-k+1} x_{t}&+R_2^{l+1}x_{n-l-k+2}+\dots +R_{k-1}^{l+1}x_{n-l-1}+R_k^{l+1}x_{n-l}\\&=R^l_k\times L_{n-l+1}+R_1^{l}(\sum_{t=1}^{n-l-k+1} x_t+x_{n-l-k+2})+R_2^{l}x_{n-l-k+3}+\dots +R_{k-1}^{l}x_{n-l}.
\end{align*}
By observing the relationship between the expression of \(L_{n-l+1}\) and the exponential term, we can obtain that\( R^l \) satisfies the following recurrence relation, 
\[\left\{
\begin{aligned}
    R_1^{l+1}&=R_1^{l}+(k-1)R_k^{l},\\
    R_2^{l+1}&=R_1^{l}+(k-2)R_k^{l},\\
    R_3^{l+1}&=R_2^{l}+(k-3)R_k^{l},\\
    \cdots\\
    R_k^{l+1}&=R_{k-1}^{l}.
\end{aligned}\right.\]
 The transformation matrix \(A_k\) we obtained here is identical to the matrix in (\ref{matrix A_k}).

Denote \(R^l=(R_1^l,R_2^l,\dots,R_k^l)^T\), we have 
\[R^{l+1}=A_kR^l=A^l_kR^1.\]
This relationship is the same as the condition stated in the lemma \ref{lemma A}. After \(n-k\) integrations we are left with

\begin{align*}
    I_{n-k}^{(k)}=\prod _{l=1}^{n-k}\frac{1}{R^k}e^{{-R_1^{n-k+1} x_{1}-R_2^{n-k+1}x_{2}-\dots -R_{k-1}^{n-k+1}x_{k-1}-R_k^{n-k+1}x_{k}}.}
\end{align*}
Then, by integrating successively with respect to \( x_{n-k},x_{n-1},\dots x_2 \) and \( x_1 \), we obtain
\[ P_n^{(k)}=\prod_{l=1}^{n-k}\frac{1}{R^l_k} \prod_{i=1}^k\frac{1}{R^{n-k+1}_i}.\]
Note that in the second product, the terms involving \(R^{n-k+1}_k\) and \(R^{n-k+1}_{k-1} = R^{n-k+2}_k\) can be incorporated into the first product. Applying the lemma \ref{lemma A}, we complete the proof. 
\end{proof}
Now we prove the lemma.

\begin{proof}
\(R^{l}=A^{l-1}_kR^1\) tell us that \(R^{l}=A_kR^{l-1}\), that is to say

\[\left\{
\begin{aligned}
  R_1^l&=R_1^{l-1}+(k-1)R_k^{l-1},\\
    R_2^l&=R_1^{l-1}+(k-2)R_k^{l-1},\\
    R_3^l&=R_2^{l-1}+(k-3)R_k^{l-1},\\
    \cdots\\
    R_k^l&=R_{k-1}^{l-1}.
\end{aligned}
\right.\]

then we can get 
\begin{align*}
    R^{l+1}_k&=R^{l}_{k-1}\\&=R_{k-2}^{l-1}+R_k^{l-1}\\
    &=R^{l-2}_{k-3}+2R^{l-2}_k+R_k^{l-1}\\
    &\cdots\\
    &=R_{1}^{l-k+2}+\sum_{i=1}^{k-2} iR_k^{l-i}.
\end{align*}
other hand, from the first two formulas, we can get
\[R_2^l-R_1^l=-R_k^{l-1},\]
which means 
\[R_1^l=R_k^{l-1}+R_2^l.\]
Then we get 
\begin{align*}
    R_k^{l+1}&=R_1^{l-k+2}+\sum_{i=1}^{k-2}iR_k^{l-i}\\
    &=R_k^{l+1-k}+R_2^{l-k+2}+(k-2)R^{l-k+2}_k+\sum_{i=1}^{k-3} iR_k^{l-i}\\
    &=R_k^{l+1-k}+R_k^{l-k+2}+R_3^{l-k+3}+\sum_{i=1}^{k-3} iR_k^{l-i}\\
    &\cdots\\
    &=\sum_{i=1}^{k} R_k^{l+1-i} .
\end{align*}
Let \(R^{-k+2}=(1,0,0,\dots,0)^T\in \mathbb{R}^k\), then \[R^{t+2-k}=A^t_k R^{2-k}=(1,1,\dots,1,0,0,\dots,0)^T\in \mathbb{R}^k, t<{k-1}.\]
the elements from the first to the \(t+1\)-th position of the vector are all \(1\), and the remaining elements are \(0\). and 
\[R^{1}=A_k^{k-1}R^{2-k}.\]
\(A_k\) is invertible, which show that \(R_k^{-k+2}=\cdots=R_k^{0}=0\). Then this proves that \( R_k^l \) \((l=1,2,\dots)\) is precisely the \(k\)-step Fibonacci-type recurrence by definition. That's to say, \( R_k^l=F_k^l \) \((l=1,2,\dots)\).

Now we  use \(R_k^{m}\) \((m=1,2,\dots)\) to present \(R_i^{l}\). Come back to the relationship,
\[\left\{
\begin{aligned}
  R_1^l&=R_1^{l-1}+(k-1)R_k^{l-1},\\
    R_2^l&=R_1^{l-1}+(k-2)R_k^{l-1},\\
    R_3^l&=R_2^{l-1}+(k-3)R_k^{l-1},\\
    \cdots\\
    R_k^l&=R_{k-1}^{l-1}.
\end{aligned}
\right.\]

Firstly, we get 
\[R_{k-1}^l=R_k^{l+1}=F_k^{l+1},\]
and \(R_{k-2}^l=R_{k-1}^{l+1}-R_{k}^l=R_k^{l+2}-R_k^l\), similarly
\[R_{k-3}^l=R^{l+1}_{k-2}-2R_k^l=R_k^{l+3}-R^{l+1}_k-2R_k^l,\]
and
\[R_{k-4}^l=R_{k-3}^{l+1}-3R_k^l=R^{l+4}_k-3R_k^l-2R_k^{l+1}-R^{l+2}_k.\]
Since
\[R_{i}^l=R_{i+1}^{l+1}-(k-i-1)R_{k}^l,i\leq k-1.\]
By induction, we obtain
\begin{align*}
    R_{i}^l&=R^{l+k-i}_k-\sum_{j=1}^{k-i-1}(k-i-j)R_k^{l+j-1}\\ &=F^{l+k-i}_k-\sum_{j=1}^{k-i-1}(k-i-j)F_k^{l+j-1},i\leq k-2.
\end{align*}
\end{proof}

Example: \(k=4\), that is, the case of a pentagon, 
then
\begin{align*}
    P^{(4)}_n&=\int_0^\infty \int_0^\infty\dots\int_{3\sum_{ l=1}^{n-4} x_{l}+2x_{n-3} +x_{n-2}}^\infty e^{-\sum_{i=1}^n x_i}dx_n\dots dx_2dx_1\\
    &=\int_0^\infty\int_0^\infty\cdots \int^\infty_{3\sum_{ l=1}^{n-5} x_{l}+2x_{n-4}+x_{n-3}} e^{-4\sum_{ l=1}^{n-4} x_l+3x_{n-3} +2x_{n-2}+x_{n-1}}dx_{n-1}\dots dx_2dx_1.
\end{align*}
that means
\[\left\{\begin{aligned}
    R_1^{l+1}&=R_1^l+3R_k^l,\\
    R_2^{l+1}&=R_1^l+2R_k^l,\\
    R_3^{l+1}&=R_2^l+R_k^l,\\
    R_k^{l+1}&=R_3^l.
\end{aligned}
\right.
\]
and 
\[ A_4=\begin{pmatrix}
1& 0&0&3\\
1& 0&0&2\\
0&1&0&1\\
0& 0&1&0
 \end{pmatrix}\]
then  
\[P_n^{(4)}=\frac{1}{(R^{n-1}_4-R^{n-3}_4)(R^n_4-R^{n-2}_4-2R^{n-3})}\prod_{i=1}^{n-2}\frac{1}{R^i_4} \]
and \(R_4^i\), \(i=1,2,\dots\) is Tetranacci numbers, with \(R_4^i=R_4^{i-1}+R_4^{i-2}+R_4^{i-3}+R_4^{i-4}\) and \(R_4^1=R_4^2=1\), \(R_4^3=2\), \(R_4^3=4\). One may verify the first few values by direct integration for small \(n\).

\section{Conclusion}
We have shown that the probability that $n$ independent uniform \([0,1]\) lengths contain no \((k+1)\) that can form a \((k+1)\)-gon is governed by a $k$-step Fibonacci‑type recurrence.  This completely resolves the generalization suggested in \cite{Sudbury2025} and reveals a direct combinatorial‑geometric meaning of these recurrences in the context of random lengths.

\bibliographystyle{plain} 
\bibliography{re}
 
\end{document}